\documentclass[11pt, a4paper]{amsart}
\usepackage{amssymb}
\usepackage{mathtools}
\usepackage{color}
\usepackage{graphicx}
\usepackage{nicefrac}
\def\O{\Omega}
\def\di{\displaystyle}

\def\R{{\mathbb {R}}}

\def\A{{\mathcal{A}}}

\def\ldos{L^2(\Omega)}
\def\lp{L^{p(\cdot)}(\Omega)}

\def\lpe{L^{p^*(\cdot)}(\Omega)}
\def\wp{W^{1,p(\cdot)}(\Omega)}
\def\wpc{W_0^{1,p(\cdot)}(\Omega)}

\def\ldosb{L^{2}(\partial\Omega)}
\def\ldosg{L^{2}(\Gamma_{int})}

\def\wpc{W_0^{1,p(\cdot)}(\Omega)}

\def\th{\mathcal{T}_h}

\def\eh{\mathcal{E}_h}

\def\wpt{W^{1,p(\cdot)}(\th)}

\def\gi{\Gamma_{int}}
\def\lpi{L^{p(\cdot)}(\Gamma_{int})}
\def\l2i{L^{2}(\Gamma_{int})}

\def\[{[\hspace{-0.05cm}[}
\def\]{]\hspace{-0.05cm}]}
\def\l2{\rceil\!|}
\def\r2{|\!\lceil}
\newtheorem*{teo*}{Theorem}
\newtheorem{teo}{Theorem}[section]
\newtheorem{lema}[teo]{Lemma} 
\newtheorem{prop}[teo]{Proposition}

\newtheorem{hyp}[teo]{Hypothesis}
\theoremstyle{definition}
\newtheorem{defi}[teo]{Definition}
\theoremstyle{remark}
\newtheorem{remark}[teo]{Remark}

\usepackage{tikz}
\usepackage{pgfplots}
\renewcommand{\div}{\operatorname{div}}
\DeclareMathOperator*{\dom}{dom}
\DeclareMathOperator*{\diam}{diam}

\DeclareMathOperator*{\esssup}{ess\,sup}
\DeclareMathOperator*{\essinf}{ess\,inf}

\begin{document}

\title{The decomposition-coordination method for the $p(x)$-Laplacian}

\author[L. M. Del Pezzo  \& S. Mart\'{i}nez ]
{Leandro Del Pezzo and Sandra Mart\'{\i}nez}

\address{Departamento  de Matem\'atica, FCEyN
\hfill\break\indent UBA (1428) Buenos Aires, Argentina.}
\email{{ ldpezzo@dm.uba.ar 
\\smartin@dm.uba.ar}
 }

\thanks{Supported by ANPCyT PICT 2006-290, UBA X078, UBA X117 . The authors are members of CONICET}

%%%%%%%%%%%%%%%%%%%%%%%%%%%%%%%%%%%%%%%%%%%%%%%%%%%%%%%%%%%%
%%%%%%%%%%%%%%%%%%%%%%%%%%%%%%%%%%%%%%%%%%%%%%%%%%%%%%%%%%%%
\begin{abstract}
	In this paper we construct two algorithms to 
	approximate the minimizer 
	of a discrete functional which comes from using 
	a discontinuous Galerkin method for a variational
	problem related to the $p(x)$-Laplacian. 
	%The function
	%$p:\Omega \to [p_1, p_2]$  is log-Holder continuous, 
	%$1 < p_1\leq p_2 \leq  2$ and $\Omega \subset \R^N$.
	We also make some numerical experiments in dimension
	two.
	%This example is motivated by its applications to image
	%processing.
\end{abstract}
%%%%%%%%%%%%%%%%%%%%%%%%%%%%%%%%%%%%%%%%%%%%%%%%%%%%%%%%%%%%
%%%%%%%%%%%%%%%%%%%%%%%%%%%%%%%%%%%%%%%%%%%%%%%%%%%%%%%%%%%%
\maketitle
%%%%%%%%%%%%%%%%%%%%%%%%%%%%%%%%%%%%%%%%%%%%%%%%%%%%%%%%%%%%
%%%%%%%%%%%%%%%%%%%%%%%%%%%%%%%%%%%%%%%%%%%%%%%%%%%%%%%%%%%%
\section{Introduction}

This work is devoted to developing and analysing two 
algorithms to approximate 
the minimizer of a discrete functional 
which comes from using a discontinuous Galerkin method for a  nonlinear, nonhomogeneous variational 
problem. This  variational problem is related to an image processing model of Chen, Levin 
and Rao \cite{CLR}, see also \cite{BCE}.

More precisely, we consider the following nonlinear variational problem: 
\begin{equation}
	\begin{aligned} 	
		\mbox{Find } u\in\A\coloneqq&\left\{v\in\wp\colon v-u_D\in\wpc\right\}\mbox{ such that }\\
		&J(u)= \min_{v\in\A}J(v),
	\end{aligned}
	\tag{P}\label{problema}
 \end{equation}
where
$$
J(v)\coloneqq\int_{\O} |\nabla v|^{p(x)}+|v-\xi|^{2} \, dx,
$$
$\O$ is a bounded connected open set in $\R^N$ 
with Lipschitz continuous boundary, $p:\overline{\O}\to 
[p_1,p_2]$ is a $\log$-H\"{o}lder continuous function with 
$1<p_1\le p_2\leq 2,$ $u_D\in\wpc$ and $\xi\in L^2(\O).$  It 
is well-known that the functional $J$ admits a unique 
minimizer 
$u\in\A.$ For the definitions of the $\log$-H\"{o}lder 
continuous function and the variable exponent
Sobolev spaces $\wp$ and $\wpc,$ see Section 
\ref{Preliminaries}.

\medskip

Note that this functional is related to the so-called 
$p(x)-$Laplacian operator, that is
\begin{equation*}
\Delta_{p(x)}=\div(|\nabla u|^{p(x)-2}\nabla u).
\end{equation*}
This operator extends the classical Laplacian 
($p(x)\equiv2$) and the
$p-$Laplacian ($p(x)\equiv p,$ $1< p<+\infty$). The interest 
in this operator was originally motivated by the
model for electrorheological fluids, see \cite{RR,R}.
	
\medskip

In \cite{DLM}, the so-called discontinuous Galerkin method 
is considered to approximate the minimizer of 
\eqref{problema}. More precisely, the authors study the 
following discrete functional,
\begin{equation*}
	\begin{aligned}
		I_h(v_h)&:=\int_{\O} |\nabla v_h+R_h(v_h)|^{p(x)}+
		|v_h-\xi|^{q(x)}\, dx
	    +\int_{\partial\O}|v_h-u_D|^{p(x)}
	    {\bf{h}}^{1-p(x)}\,dS\\
	    &+\int_{\Gamma_{int}}|\[v_h\]|^{p(x)}
	    {\bf{h}}^{1-p(x)}\, dS,
    \end{aligned}
\end{equation*}	
where ${\bf{h}}$ is the local mesh size, 
$h$ is the global mesh size,
$\Gamma_{int}$ is the union 
of the interior edges of the elements, 
\begin{math}\[v_h\]\end{math} is the jump of the 
function between two edges, $\nabla v_h$ denotes
the elementwise gradient of $v_h$ and $R_h$ is the lifting 
operator, see Section \ref{Preliminaries} 
for a precise definition. 
Observe that the boundary condition is weakly imposed by the 
second  term of the functional.

With this setting, the discrete problem is to find a 
minimizer $u_h$ of $J_h$ over the space $S^k(\th)$ 
of all the functions that
are polynomials of degree at most $k$ in each element, with 
$k\ge1,$ see Subsection \ref{appA2} for details.

\medskip

In \cite{DLM}, the authors prove the following result.
\begin{teo}\label{conver1}
	Let  $\O$ be a polyhedral domain, $u_D\in W^{2,2}(\O),$ 
	and $u_h\in S^k(\th)$ be the minimizer of ${J}_h$ 
	for any $h\in (0,1]$. 
	If $u$ is the minimizer of ${J}$ then
  	$$
  		I_h(u_{h})\rightarrow  J(u),\quad R_h(u_h)\to 0,
  		\quad  u_{h}\to  u \mbox{ strongly in } \wp, 
  		\mbox{ and }
  	$$
  	\begin{equation*}
    	\int_{\partial\O}|u_h-u_D|^{p(x)}
    	{\bf{h}}^{1-p(x)}\,dS 
	    +\int_{\Gamma_{int}}|\[u_h\]|^{p(x)}
	    {\bf{h}}^{1-p(x)}\, dS \rightarrow  0
	\end{equation*}
\end{teo}

\medskip

Since we want to implement this method for some examples, 
the next step is to find a good approximation of the 
minimizer of the discrete functional.

The  methods for finding minimizers of functionals, 
such as the BFGS Quasi- Newton, work when the dimension of 
the space is small. However, we observe that these methods 
are to slow when we refine the mesh. We also observe, 
in some numerical experiments, that the  decomposition-- 
coordination--method (DCM), defined in \cite[Chapter VI]{G}, 
is more suitable for our problem.

\medskip

The DCM is used to approximate the minimizers of functionals 
that can be written in the form
	$$
    		J(v)=F(Bv)+G(v),
  	$$
where  $F:H\to \R$, $G:V\to \R$ are convex functions, 
$B:V\to H$ is a linear operator and $V$ and $H$ are 
topological vector spaces.

In this context, the problem of finding 
minimizers of $J$ 
over $V$ is equivalent to find $(q_0,v_0)\in 
W\coloneqq\{(q,v)\in  V\times H: Bv-q=0\}$ such that
  	\begin{equation}\label{ecudos}
      		F(q_0)+G(v_0)=\min_{(q,v)\in W} \{F(q)+G(v)\}.
   	\end{equation}

\medskip

%In the practical applications, this decomposition is useful 
%when $F$ is non-linear and $G'$ is linear, 
%due to it can be proved that, the problem is reduced to find 
%a sequence $\{(u^n,q^n)\}\subset V\times H$ such that $u^n$ 
%solves a linear differential equation and $q^n$ solves a 
%non-linear equation.

In the practical applications, 
under the following assumptions
\begin{enumerate}\label{hipotesis}
	\item[(H1)] $F : H \to \R,\ G : V \to \R $ are lower semicontinuous functions and 
		$$
			\dom(F\circ B)\cap \dom(G)\neq\emptyset;
		$$
	\item[(H2)] $F$ is a convex Gateaux--diffentiable functional 
	and 
		$$
			\lim_{|q|\to\infty} \frac{F(q)}{|q|}=\infty;
		$$
	\item[(H3)] The rank of $B$ is close in $H;$
	\item[(H4)] $B$ injective;
\end{enumerate}
In \cite{G} the authors prove that there exists
a sequence $\{(u^n,q^n)\}\subset V\times H$ such that $u^n$ 
solves a linear differential equation, $q^n$ solves a 
non-linear equation, and
	\begin{align*}
		u^n\to v_0=u \quad &\mbox{ strongly in } V\\
		q^n\to q_0=B(u)\quad &\mbox{ strongly in } H
	\end{align*}
where $u$ is the minimizer of $J$.

\medskip

If we write the functional $I_h$ in this form, we have that
\begin{equation*}
	I_h(v)=F(Bv)+G(v)
\end{equation*}
where here $V=S^k(\th)$, $H=S^l(\th)\times S^l(\th)$, $k,l\in\mathbb{N}_0$ with $l\geq k-1$, 
$Bv=R_h(v)+\nabla v$,
\begin{align*}
	F(q)&=\int_{\O} |q|^{p(x)} \, dx \quad \mbox{ and }\\
	G(v)&=\int_{\O}|v-\xi|^{q(x)}\, dx
	+\int_{\partial\O}|v-u_D|^{p(x)}{\bf{h}}^{1-p(x)}\,dS
	+\int_{\Gamma_{int}}|\[v\]|^{p(x)}{\bf{h}}^{1-p(x)}\, dS.
\end{align*}

We can observe that, (H1), (H2), (H3) hold, but (H4) does not hold, that is $B$ is not injective.
Moreover, $G'$ is not linear.

To prove the convergence of the method,  it is not necessary 
the assumption that $G'$  is linear, but it is useful for the implementation. For this reason and 
since we are interested in the case   $p(x)\leq 2,$ we define  a new discrete functional
\begin{equation*}
	J_h(v)=F(Bv)+G(v)
\end{equation*}
where now 
\begin{equation*}
	G(v)= \int_{\O}|v_h-\xi|^{2} \, dx
	+\int_{\partial\O}|v_h-u_D|^{2}
	{\bf{h}}^{\nicefrac{-2}{p'(x)}}\,dS
	+\int_{\Gamma_{int}}|\[v_h\]|^{2}
	{\bf{h}}^{\nicefrac{-2}{p'(x)}}\, dS,
\end{equation*}
and $F$ and $B$ are defined as before.
In this manner, $G'$ is linear. Observe that, here we have 
to change the power over the 
function $\bf{h}$.

To overcome the lack of  injectivity of the functional $B$, we will use that our functional 
$G$ is Gateaux-differentiable and convex.

\medskip

Now, we are ready to state the main results of this paper.

\medskip

Since we change the discrete functional we have to prove a result similar to Theorem
\ref{conver1}. More precisely, we prove the following theorem.

\begin{teo}\label{conv}
	Let  $\O$ be a polygonal domain in $\R^N$, 
	$p:\overline{\O}\to [p_1,2]$ ($\nicefrac{N}2<p_1\leq 2$) be a $\log$-H\"{o}lder continuous
	and $u_D\in W^{2,2}(\O)$.
	For each $h\in (0,1]$, let $u_h\in S^k(\th)$ be the minimizer of ${J}_h$. If $u$ is the minimizer of ${J}$ then
	\begin{align*}
		u_{h}&\to u \mbox{ strongly in } L^{s(\cdot)}(\O) \quad \forall s\in \mathcal{K},\\
		u_{h}&\to u \mbox{ strongly  in } L^{2}(\partial\O),\\
		 J_h(u_{h})&\rightarrow  J(u),\\
		 R_h(u_h)&\to 0,\\
		\int_{\partial\O}|u_h-u_D|^{2}{\bf{h}}^{\nicefrac{-2}{p'(x)}}\,dS &
		+\int_{\Gamma_{int}}|\[u_h\]|^{2}{\bf{h}}^{\nicefrac{-2}{p'(x)}}\, dS \rightarrow  0,\\
		\nabla u_{h}&\to \nabla u \mbox{ strongly in } \lp,
	\end{align*}
	where $\mathcal{K}=\{s \in L^{\infty}(\O)\colon 1 \le s(x) < p^*(x) -\varepsilon \mbox{ for some }\varepsilon > 0\}$.
\end{teo}

\medskip

We define two algorithms that construct a sequence $\{u_h^n\}_{n\in\mathbb{N}}$ that approximates, for each $h\geq 0$ the  minimizers of $J_h$ and finally we prove the convergence of both algorithms.

\begin{teo}\label{conalg1}
	Let $h\geq 0$ and  $ (u_h, \eta_h, \lambda_h)\in V\times H\times H$ be a saddle-point of  $\mathcal{L}_r$. If
	\begin{equation}\label{condrho}
		0<\alpha_0\leq \rho_n\leq \alpha_1<2r
	\end{equation}
	and $(u_h^n,\eta_h^n,\lambda_h^n)\in V \times H \times H$ is the  solutions given by Algorithm 1 then
	\begin{align*}
		u_h^n\to u_h &\mbox{ in } V, \\
		\eta_h^n\to \eta_h=Bu_h &\mbox{  in } H,\\
		\lambda_h^{n+1}-\lambda_h^n \to 0 &\mbox{  in } H,
	\end{align*}
$$
	\lambda_h^n  \mbox{ is bounded.}
$$		
Moreover,
	\begin{align*}
		\[u_h^n-u_h\]\to 0 &\mbox{  in } \ldosg,\\
		u_h^n \to u_h &\mbox{ in } \ldosb,\\
		R(u_h^{n})\to R(u_h) & \mbox{  in } H,\\
		\nabla u_h^{n} \to \nabla u_h  &\mbox{  in } H.
	\end{align*}
\end{teo}

\begin{teo}\label{conalg2}
	Let $h\geq 0$ and  $ (u_h, \eta_h, \lambda_h)\in V\times H\times H$ be a saddle-point of  $\mathcal{L}_r$. If
	\begin{equation}\label{condrho1}
		0<\rho_n=\rho<r\frac{(1+\sqrt{5})}{2}
	\end{equation}
	and $(u_h^n,\eta_h^n,\lambda_h^n)\in V \times H \times H$ is the  solutions given by Algorithm 2 then
	\begin{align*}
		\eta_h^n\to \eta_h=Bu_h &\mbox{  in } H ,\\
		\lambda_h^{n+1}-\lambda_h^n \to 0 &\mbox{  in } H,
	\end{align*}
	$$
	\lambda_h^n  \mbox{ is bounded }.
	$$	
	Moreover,
	\begin{align*}
		u_h^n\to u_h &\mbox{  in } V,\\ 
		\[u_h^n-u_h\]\to 0 &\mbox{  in } \ldosg,\\
		u_h^n \to u &\mbox{  in } \ldosb,\\
		R(u_h^{n})\to R(u_h) & \mbox{  in } H,\\
		\nabla u_h^{n} \to \nabla u_h  &\mbox{  in } H.
	\end{align*}
\end{teo}

\medskip

Let us end the introduction with a brief comments on previous bibliography. In \cite{BDS,DM2},
the convergence of conforming finite element method approximations for the Dirichlet problem of the $p(\cdot)-$Laplacian
is studied. Moreover, in \cite{DM2}, we study the convergence rate using the regularity results obtained in \cite{DM1}.

Finally, we want to mention that in \cite{BCE} and \cite{CLR} the authors find an approximation of the solutions 
by using an explicit finite difference scheme for the associated parabolic problem.

\medskip

{\bf Outline of the paper.} In Section \ref{Preliminaries}, we state several properties of the variable
exponent Sobolev spaces, we give some definitions and properties related to the mesh and to
the broken Sobolev spaces; In section \ref{contg}, we prove Theorem \ref{conv};
In section \ref{dcor}, the decomposition-coordination method is studied and the convergence of the algorithms are proved;
Finally, in Section \ref{numerical}, we give some numerical examples.
%%%%%%%%%%%%%%%%%%%%%%%%%%%%%%%%%%%%%%%%%%%%%%%%%%%%%%%%%%%%
%%%%%%%%%%%%%%%%%%%%%%%%%%%%%%%%%%%%%%%%%%%%%%%%%%%%%%%%%%%%
\section{Preliminaries} 
\label{Preliminaries}
We begin with a review of the basic results that will be needed in subsequent sections. 
The results are generally stated without proof, although we attempt to provide 
good references where the proofs can be found. Also, we introduce some of our notational 
conventions.

\subsection{The spaces $L^{p(\cdot)}(\Omega)$ and $W^{1,p(\cdot)}(\Omega)$}
We first introduce the spaces  $L^{p(\cdot)}(\Omega)$ and $W^{1,p(\cdot)}(\Omega)$ 
and state some of their properties.

\medskip

Let $p \colon\Omega \to  [p_1,p_2]$ be a measurable bounded
function, called a variable exponent on $\Omega$ where
$p_{1}\coloneqq \essinf \,p(x)$ and $p_{2} \coloneqq \esssup \,p(x)$ 
with $1\le p_1\leq p_2<\infty$.

We define the variable exponent Lebesgue space
$L^{p(\cdot)}(\Omega)$ to consist of all measurable functions 
$u\colon\Omega \to \R$ for which the modular
$$
	\varrho_{p(\cdot)}(u) \coloneqq \int_{\Omega} 
	|u(x)|^{p(x)}\, dx
$$
is finite. We define the Luxemburg norm on this space by
$$
	\|u\|_{L^{p(\cdot)}(\Omega)} = \|u\|_{p(\cdot)} 
	\coloneqq 
	\inf\{k >0\colon \varrho_{p(\cdot)}(u/k)\leq 1 \}.
$$
This norm makes $L^{p(\cdot)}(\Omega)$ a Banach space.

\medskip

The following properties can be obtained directly from the
definition of the norm. For the proof see \cite[Theorem 1.3 and Theorem 1.4 ]{FS}.

\begin{prop}\label{equi}
If $u,u_n\in L^{p(\cdot)}(\Omega)$, 
$\|u\|_{p(\cdot)}=\lambda,$
then
\begin{enumerate}
	\item $\lambda<1$ $(=1, >1)$ iff 
		$\di\int_{\O}|u(x)|^{p(x)}\, dx<1\ (=1,>1);$ 
	\item If $\lambda\geq1,$ then 
		$\di\lambda^{p_1}\le \int_{\O} |u(x)|^{p(x)}\, 
		dx\leq\lambda^{p_2};$ 
	\item If $\lambda\leq 1,$ then
		$\di\lambda^{p_2}\le \int_{\O} |u(x)|^{p(x)}\, 
		dx\leq
		\lambda^{p_1};$ 
	\item $\di\int_{\O}|u_n(x)|^{p(x)}\, dx \to 0$ iff
		$\di\|u_n\|_{p(\cdot)}\to 0.$ 
\end{enumerate}
\end{prop}

For the proofs of the following two theorems we refer the 
reader to \cite{KR}.

\begin{teo}\label{imb}
	Let $q(x)\leq p(x)$, then
 	$L^{p(\cdot)}(\Omega)
 	\hookrightarrow L^{q(\cdot)}(\Omega)$
	continuously.
\end{teo}

\begin{teo}\label{ref}
	Let $\di p'(x)$ such that, ${1}/{p(x)}+{1}/{p'(x)}=1.$ 
	Then
	$L^{p'(\cdot)}(\Omega)$ is the dual of 
	$L^{p(\cdot)}(\Omega)$.
	Moreover, if $p_1>1$, $L^{p(\cdot)}(\Omega)$ and
	$W^{1,p(\cdot)}(\Omega)$ are reflexive.
\end{teo}

Now we give  some well known inequalities.

\begin{prop}\label{desip} 
	For any $x$ fixed we have the following inequalities
	\begin{align*}
		&|\eta-\xi|^{p(x)}\leq C (|\eta|^{p(x)-2} \eta-
		|\xi|^{p(x)-2} \xi)(\eta-\xi) &\quad 
		\mbox{ if } p(x)\geq 2,\\
		&|\eta-\xi|^2\Big(|\eta|+|\xi|\Big)^{p(x)-2}
		\leq C (|\eta|^{p(x)-2} \eta-|\xi|^{p(x)-2} \xi)
		(\eta-\xi)&\quad  \mbox{ if } p(x)< 2,\\
		&|\eta|^{p(x)}\leq2^{p(x)-1}
		(|\eta-\xi|^{p(x)}+|\xi|^{p(x)})
		&\quad  \mbox{ if }p(x)\geq 1.
	\end{align*}
\end{prop}

The following properties will be used throughout the paper.

\begin{prop}\label{debil}
	Let $F_n,F\in\lp.$
	\begin{enumerate}
	\item If
		$$
			F_n\rightharpoonup F \mbox{ weakly in } \lp
		$$
		then
		$$
			\int_{\O}|F|^{p(x)}\, dx \leq 
			\liminf_{n\to\infty}
			\int_{\O}|F_n|^{p(x)}\, dx.
		$$
	\item If
 		$$ 
			F_n \rightarrow F \mbox{ strongly in } \lp
		$$ 
		then
 		$$
			\int_\O |F_n|^{p(x)}\, dx \to \int_\O 
			|F|^{p(x)}\, dx.
		$$
	\item If	 
		$$
			F_n\rightharpoonup F \mbox{ weakly in } \lp 
			\quad \mbox{and}\quad\int_\O |F_n|^{p(x)}\, 
			dx \to \int_\O |F|^{p(x)}\, dx
		$$
		then
 		$$ 
 			F_n \rightarrow F \mbox{ strongly in } \lp.
 		$$
	\end{enumerate}
\end{prop}

\begin{proof}
For the proof of (1) and (3) see \cite[Theorem 3.9 and 
Lemma 2.4.17]{DHHR}. Finally (2) follows by 
\cite[Proposition 2.3]{FZ}.
\end{proof}

Let $W^{1,p(\cdot)}(\Omega)$ denote the space of measurable
functions $u$ such that, $u $ and the distributional derivative
$\nabla u$ are in $L^{p(\cdot)}(\Omega)$. The norm
$$
	\|u\|_{W^{1,p(\cdot)}(\O)}=\|u\|_{1,p(\cdot)}\coloneqq 
	\|u\|_{p(\cdot)} + \| |\nabla u|\|_{p(\cdot)}
$$
makes $W^{1,p(\cdot)}(\Omega)$ a Banach space.

We define the space $W_0^{1,p(\cdot)}(\Omega)$ as the closure of
$C_0^{\infty}(\Omega)$ in $W^{1,p(\cdot)}(\Omega)$. 

\medskip

We now introduce the most important condition on the exponent
in the study of variable exponent spaces, the $\log$-H\"older 
continuity condition.

\begin{defi}
	We say that a function  $\alpha:\overline{\O}\to\R$ is $\log$-H\"{o}lder continuous if 
	there exists a constant $C_{log}$ such that
	$$
		|\alpha(x) - \alpha(y)| \leq 
		\frac{C_{log}}{\log\,
		\left(e+\frac{1}{|x - y|}\right)}
		\quad \forall\, x,y\in\overline{\O}.
	$$	
\end{defi}

For example, it was proved
in  \cite[Theorem 3.7]{D}, that if one assumes that $\partial\O$
is Lipschitz and   $p:\overline{\O}\to[1,+\infty)$ is $\log$-H\"{o}lder continuous then
$C^{\infty}(\bar{\Omega})$ is dense in $W^{1,p(\cdot)}(\Omega).$
See also  \cite{Di,DHN, Fanx2,KR,Sam1}. 

\medskip

We now state two Sobolev embedding Theorems.
Here, $p^*$ and $p_*$ are the Sobolev critical exponents for these
spaces, i.e.
\begin{equation*}
	%\label{criticos}
	p^*(x)\coloneqq
		\begin{cases}
			\dfrac{p(x)N}{N-p(x)} & \mbox{if } p(x)<N,\\
			+\infty & \mbox{if } p(x)\ge N,
		\end{cases}
	\quad \mbox{ and }
	\quad p_*(x)\coloneqq
		\begin{cases}
			\dfrac{p(x)(N-1)}{N-p(x)} & \mbox{if } p(x)<N,\\
			+\infty & \mbox{if } p(x)\ge N.
		\end{cases}
\end{equation*}

For the proofs of the following theorems see
\cite{D3} and \cite[Corollary 2.4]{Fanx}, respectively.

\begin{teo}\label{embed}
	Let $\Omega$ be a Lipschitz domain. Let 
	$p:\overline{\O}\to [1,\infty)$
	be a $\log$-H\"{o}lder continuous function. 
	Then the embedding
	$\wp\hookrightarrow \lpe$ is continuous.
\end{teo}

\begin{teo}\label{traza}
	Let $\Omega$ be a bounded  Lipschitz domain.
	Suppose that $p\in C^0(\overline{\O})$ with $p_1>1$. 
	If $r\in C^0(\partial\O)$ satisfies the condition 
	$1\leq r(x)<p_* (x)$ for all $x\in\partial\O,$ then there is 
	a compact boundary trace embedding 
	$\wp\hookrightarrow L^{r(\cdot)}(\partial\O)$.
\end{teo}

%%%%%%%%%%%%%%%%%%%%%%%%%%%%%%%%%%%%%%%%%%%%%%%%%%%%%%%%%%%%%%%%%%%
\subsection{The mesh $\th$ and properties of $\wpt$}
\label{appA2} We now give some definitions and properties related to
the mesh and to the broken Sobolev space.

\begin{hyp}\label{omega}
	Let $\O$ be a bounded polygonal domain and $(\th)_{h\in(0,1]}$
	be a family of partitions of $\overline{\O}$ into polyhedral
	elements. We assume that there exists a finite number of reference
	polyhedral $\hat{\kappa}_1,\dots,\hat{\kappa}_r$ such that for all
	$\kappa\in \th$ there exists an invertible affine map $F_{\kappa}$
	such that, $\kappa=F_{\kappa}(\hat{\kappa}_i)$. We assume that
	each $\kappa\in\th$ is closed and that $diam(\kappa)\leq h$ for
	all $\kappa\in \th$.
\end{hyp}

Now we give some notation,
\begin{align*}
	\mathcal{E}_h&\coloneqq\{\kappa\cap\kappa'\colon \dim_H(\kappa\cap\kappa')=N-1\}
	\cup\{\kappa\cap\partial\O\colon \dim_H(\kappa\cap\partial\O)=N-1\},\\
	\Gamma_{int}&\coloneqq\bigcup\{ e\in\eh	\coloneqq \dim_H(e\cap\partial\O)<N-1\},
\end{align*}
where $\dim_H$ is the Hausdorff dimension.

\medskip

We also assume that the mesh satisfies the following hypotheses.

\begin{hyp}\label{mesh}  
	The family of partitions $(\th)_{h\in(0,1]}$
    	satisfies the Hypothesis \ref{omega} and
	\begin{enumerate}
		\item[(a)] There exist positive constants  $C_1$ 
			and $C_2$,independent of $h$, such that for each 
			element $\kappa\in\th$
			$$
				C_1h_{\kappa}^N\le|\kappa|\le C_2 
				h_\kappa^N.
			$$
		\item[(b)]There exists a constant $C_1>0$ such that 
			for all 
			$h\in(0,1]$ and for all face $e\in\eh$ there 
			exists a point $x_e\in e$ and a radius
    		$\rho_e\geq C_1 \diam(e)$ such that 
    		$B_{\rho_e}(x_e)\cap A_e\subset e$, where
  		 	$A_e$ is the affine hyperplane spanned by $e$.
    		Moreover, there are positive constants such that
        	$$
				ch_{\kappa}\leq h_e \leq C h_{\kappa}, 
				\quad c h_{\kappa'}\leq h_e\leq C 
				h_{\kappa'}
			$$
        	where $e=\kappa\cap\kappa'$.
	\end{enumerate}

\end{hyp}

\medskip

Now, we introduce  the finite element spaces associated with
$\th.$ We define the variable broken Sobolev space as
$$
	W^{1,p(\cdot)}(\th)\coloneqq\{u\in L^1(\Omega)
	\colon u|_{\kappa}\in
	W^{1,p(\cdot)}(\kappa) \mbox{ for all } \kappa \in\th\},
$$
and the subspaces
\begin{align*}
	U^k(\th)\coloneqq &
	\{u\in C(\O)\colon u|_{\kappa}\in P^k 
	\mbox{ for all }\kappa \in\th\},\\
	S^k(\th)\coloneqq&\{u\in L^1(\Omega)\colon 
	u|_{\kappa}\in P^k \mbox{ for all } \kappa \in\th\},
\end{align*}
where $P^k$ is the space of polynomials functions of degree at most $k.$

\medskip

For each face $e\in\mathcal{E}_h,$  $e\subset\Gamma_{int}$ 
we denote by $\kappa^+$ and $\kappa^-$ its neighboring 
elements. We write $\nu^+,\nu^-$ to denote the outward 
normal unit vectors to the boundaries 
$\partial\kappa^{\pm}$, respectively. The jump of a
function $u\in\wpt$ and the average of a vector-valued 
function
$\phi\in(\wpt)^N,$ with traces $u^{\pm},$ $\phi^{\pm}$ from
$k^{\pm}$ are, respectively, defined as the vectors
\begin{equation*}
	\[u\]\coloneqq u^+ \nu^++u^- \nu^- \quad
	\mbox{ and }\quad
    \{\phi\}\coloneqq\frac{\phi^++ \phi^-}{2}.
\end{equation*}

Let ${\bf h}\colon\partial\Omega\cup \Gamma_{int}\to\R$ a piecewise
constant function define by
$$
	{\bf h}(x)\coloneqq\diam(e)  \quad \mbox{ if }  x\in e,
$$
where  $e\in \mathcal{E}_h$.

We consider the following seminorm on $W^{1,p(\cdot)}(\th)$,
\begin{equation*}
	|u|_{W^{1,p(\cdot)}(\th)}\coloneqq
	\|\nabla u\|_{\lp}+ \|\[u\]
	{\bf{h}}^{\frac{-1}{p'(x)}}
	\|_{L^{p(\cdot)}(\Gamma_{int})}.
\end{equation*}

%%%%%%%%%%%%%%%%%%%%%%%%%%%%%%%%%%%%%%%%%%%%%%%%%%%%%%%%%%%%%%%
\subsection{The lifting operator}\label{lifop}
Finally we define, as in \cite{BO} (see also
\cite{ABCM}), the lifting operator.
\begin{defi}\label{lifting}
	Let $l\ge0$ and
 	$R_h\colon\wpt\to S^l(\th)^N$ defined as,
 	\begin{equation*}
		\int_{\O} \langle R_h(u), \phi\rangle\, dx \coloneqq
		-\int_{\Gamma_{int}} \langle \[u\], \{\phi\} 
		\rangle \, dS 
	\end{equation*}
	for all $\phi\in S^l(\th)^N.$
\end{defi}

This operator appears in the first term of the discretized
functional $J_h$. As we can see from the definition, this 
operator represents the contribution of the jumps to the 
distributional gradient. This is the reason  why it is 
crucial to add  this term in order to have the consistency 
of the method.

We point out that this lifting operator was first used in
\cite{BR1997} in order to describe the contributions of  the 
jumps across the interelements of the computed solution on 
the (computed) gradient of the solution in a mixed 
formulation of
Navier-Stokes equations. It was also used in 
\cite{BMMPR2000}
where a solid mathematical background for the method 
introduced in
\cite{BR1997} was proposed.

\medskip

Now, we state a bound of the $\lp$-norm of $R_h(u)$ in terms 
of the jumps of $u$ in $\gi$. For the proof see \cite{DLM}.

\begin{lema}\label{lifcont1}
	Let $p:\overline{\O}\to\colon[1,\infty)$ be a $\log$-
	H\"{o}lder continuous in
	$\O$. Then, there exists a constant $C$ such that,
	\begin{equation*}
		\|R_h(u)\|_{\lp}\leq C\|{\bf{h}}^{-1/p'(x)}\[u\]
		\|_{\lpi}
		\quad \forall u\in\wpt\quad \forall h\in (0,1].
	\end{equation*}
\end{lema}
%%%%%%%%%%%%%%%%%%%%%%%%%%%%%%%%%%%%%%%%%%%%%%%%%%%%%%%%%%%%
%%%%%%%%%%%%%%%%%%%%%%%%%%%%%%%%%%%%%%%%%%%%%%%%%%%%%%%%%%%
\section{Convergence of the discontinuous Galerkin FEM}\label{contg}
In this section we prove the convergence of the discontinuous Galerkin FEM.

\medskip

From now on, we make the following assumption: 
Let  $\O$ be a bounded polygonal domain in $\R^N$ and  
$p:\overline{\O}\to [p_1,2]$ ($1<p_1\leq 2$) be a $\log$-H\"{o}lder continuous function.

\medskip

Our next result follows by using Lemma \ref{lifcont1} and the fact that $L^2(\Gamma_{int})\subset \lpi$.

\begin{lema}\label{lifcont}
	There exists a constant $C$ such that
	\begin{equation*}
		\|R_h(v)\|_{\lp}\leq C\|{\bf{h}}^{-1/p'(x)}\[v\]
		\|_{L^2(\Gamma_{int})}
	\end{equation*}
	for all $v\in\wpt$ and for all $h\in(0,1].$
\end{lema}

Now, we prove the coercivity of the functional.

\begin{teo}\label{cm}
	For each $h\in(0,1],$ let  $v_h\in \wpt$. 
	If there exists a constant $C$ independent of $h$ 
	such that
	${J}_h(v_h)\le C$ for all $h\in (0,1]$,
	then
	$$
		\sup_{h\in(0,1]}\left(\|v_h\|_{L^1(\Omega)}
		+|v_h|_{\wpt}\right)<\infty.
	$$
	Moreover,
	$$
		\sup_{h\in(0,1]}\int_{\partial\O}|v_h-u_D|^{p(x)}
		{\bf{h}}^{1-p(x)}\,dS<\infty.
	$$
\end{teo}
\begin{proof}
	Since $J_h(v_h)\leq C,$ we have that
	\begin{equation*}
		\int_{\Gamma_{int}}|\[v_h\]|^{2}
		{\bf{h}}^{\nicefrac{-2}{p'(x)}}\, dS\leq C
	\end{equation*}
	then, by Lemma \ref{lifcont}, $\|R_h(v_h)\|_{\lp}$ 
	is bounded. Therefore, using Proposition \ref{desip},
	we have 
	\begin{equation*}
		J_h(v_h)+C\geq C\int_{\O} |\nabla v_h|^{p(x)}\, dx+
		\int_{\partial\O} |v_h-u_D|^2
		 {\bf{h}}^{-\nicefrac{2}{p'(x)}}\, dS
		 + \int_{\Gamma_{int}}|\[v_h\]|^{2}
		 {\bf{h}}^{\nicefrac{-2}{p'(x)}}\, dS.
	\end{equation*}

	By the above inequality and the fact that 
	$L^2\subset L^{p(\cdot)}$ we get
	\begin{align*}
		\int_{\O}|\nabla v_h|^{p(x)}\,dx &\leq C,\\
		\int_{\partial\O} |v_h-u_D|^{p(x)} 
		{\bf{h}}^{1-p(x)}\, dS &\leq C,\\
		\int_{\Gamma_{int}}|\[v_h\]|^{p(x)}
		{\bf{h}}^{1-p(x)}\, dS &\leq C.
	\end{align*}
	Finally, the proof follows as in the end 
	of the proof of Theorem 6.2 in \cite{DLM}.
\end{proof}

The following theorem was proved in \cite{DLM}.

\begin{teo}\label{fusion}
	Let $u_h\in S^k(\th)$ be such that
	$$
		\sup_{h\in(0,1]}\left(\|u_h\|_{L^1(\Omega)}+|u_h|_{\wpt}\right)<\infty
		\quad\mbox{ and }\quad
		\sup_{h\in(0,1]}\int_{\partial\O}|u_h-u_D|^{p(x)}{\bf{h}}^{1-p(x)}\,dS<\infty.
	$$
	Then, there exist a sequence
	$h_j\to0$ and a function $u\in\wp$ such  that
	\begin{align*}
    		u_{h_j}&\stackrel{*}{\rightharpoonup} u\quad\ \ \mbox{weakly* in } BV(\O)\\
    		\nabla u_{h_j}+R_h(u_{h_j})&\rightharpoonup \nabla u \quad\mbox{weakly in } \lp,\\
     		u_{h_j}&\rightarrow u\quad\ \ \mbox{strongly in } L^{p(\cdot)}(\partial\O),\\
    		u_{h_j}&\rightarrow u\quad\ \ \mbox{strongly in } L^{s(\cdot)}(\O)\quad \forall s\in\mathcal{K},
    		%\label{lqb} u_{h_j}&\rightarrow   u\quad\ \ \mbox{strongly in } L^{t(\cdot)}(\partial\O).
	\end{align*}
	where $\mathcal{K}=\{s \in L^{\infty}(\O)\colon 1 \le s(x) < p^*(x) -\varepsilon \mbox{ for some }\varepsilon > 0\}$.
\end{teo}

Before proving the convergence of the sequence of mimizers, we need an auxiliary lemma.

\begin{lema}\label{auxiliar}
	Let $h\in (0,1]$ and $u_D\in W^{2,2}(\O).$ 
	If $v\in W^{2,2}(\O)\cap\A,$ then there exists $v_h\in U^1(\th)$ such that
	$$
	\di \|v_h-v\|_{H^1(\O)}\to 0 \quad \mbox{ as } h\to 0,
	$$
	and $$J_h(v_h)\to J(v)\quad \mbox{ as } h\to 0.$$
\end{lema}

\begin{proof}
	Given $v\in W^{2,2}(\O)\cap\A,$ by standard approximation theory (see \cite[Theorem 3.1.5]{Ci}), 
	we have that there exists $v_h\in U^1(\th)$ such that 
	$$
		\di \|v_h-v\|_{H^1(\O)}\to 0 
	$$
	as $h\to 0,$ and
	$$
		\|v-v_h\|_{\ldosb}\leq C h \|D^2 v\|_{\ldos}.
	$$
	Since $1<p_1\le p(x)\le 2,$ we have
	$$
		\int_{\partial\O} |v-v_h|^2 {\bf{h}}^{-\nicefrac{2}{p'(x)}}  \, dS\le
		C h^{\nicefrac{-2}{p_1}} \int_{\partial\O} |v-v_h|^2 \, dS 
		\leq C h^{\nicefrac{-2}{p_1}}  h^2 \|D^2 v\|^2_{\ldos}\to 0
	$$
	as $h\to0.$

	Finally, since $v_h\in \wp$, we have that $\[v_h\]=0$ and $R_h(v_h)=0$. Then, using Proposition \ref{debil}, we have 
	$$
		J_h(v_h)=\int_{\O}|\nabla v_h|^{p(x)}+|v_h-\xi|^{2}\, dx+
		\int_{\partial\O} |v_h-u_D|^2 {\bf{h}}^{\nicefrac{-2}{p'(x)}}  \, dS\to 
		\int_{\O}(|\nabla v|^{p(x)}+|v-\xi|^{2})\, dx
	$$ 
	as $h\to0.$ The proof is complete. 
\end{proof}

Now we prove Theorem \ref{conv}.

%\begin{teo}\label{conv}
%	Let  $\O$ be a polygonal domain in $\R^N$, 
%	$p:\overline{\O}\to [p_1,2]$ ($\nicefrac{N}2<p_1\leq 2$) be a $\log$-H\"{o}lder continuous
%	and $u_D\in W^{2,2}(\O)$.
%	For each $h\in (0,1]$, let $u_h\in S^k(\th)$ be the minimizer of ${J}_h$. If $u$ is the minimizer of ${J}$ then
%	\begin{align*}
%		u_{h}&\to u \mbox{ strongly in } L^{s(\cdot)}(\O) \quad \forall s\in \mathcal{K},\\
%		u_{h}&\to u \mbox{ strongly  in } L^{2}(\partial\O),\\
%		 J_h(u_{h})&\rightarrow  J(u),\\
%		 R_h(u_h)&\to 0,\\
%		\int_{\partial\O}|u_h-u_D|^{2}{\bf{h}}^{\nicefrac{-2}{p'(x)}}\,dS &
%		+\int_{\Gamma_{int}}|\[u_h\]|^{2}{\bf{h}}^{\nicefrac{-2}{p'(x)}}\, dS \rightarrow  0,\\
%		\nabla u_{h}&\to \nabla u \mbox{ strongly in } \lp,
%	\end{align*}
%	where $\mathcal{K}=\{s \in L^{\infty}(\O)\colon 1 \le s(x) < p^*(x) -\varepsilon \mbox{ for some }\varepsilon > 0\}$.
%\end{teo}

\begin{proof}[Proof of Theorem \ref{conv}] By Lemma \ref{auxiliar}, 
	there exists $w_h\in U^1(\th)$ such that $J_h(w_{h})\rightarrow  J(u_D).$ Then
	\begin{equation}
		J_h(u_{h})\leq J_h(w_{h})\leq C\quad \forall h>0.
		\label{cotaimp}
	\end{equation}
	By Theorem \ref{cm} and Theorem \ref{fusion}, 
	we obtain that there exists a subsequence $u_{h_j}$ of $u_h$ 
	such that
	\begin{align*}
		u_{h_j}&\to u \mbox{ strongly in } L^{s(\cdot)}(\O) \quad \forall s\in \mathcal{K}\\
		u_{h_j}&\to u \mbox{ strongly  in } L^{p(\cdot)}(\partial\O),\\
		\nabla u_{h_j}+R_h(u_{h_j})&\rightharpoonup \nabla u \quad\mbox{weakly in } \lp.
	\end{align*}
	
	On the other hand, by \eqref{cotaimp}
	$$
		\int_{\partial\O} |u_{h_j}-u_D|^2 
		{\bf{h_j}}^{-2/p'(x)}\, dS\leq C
	$$ 
	then
	$$
		\int_{\partial\O} |u_{h_j}-u_D|^2 dS\to 0 
		\quad\mbox{as }{j}\to+\infty,
	$$
	that is $u_{h_j}\to u_D$ strongly in $\ldosb.$
	 Therefore, since $u_{h_j}\to u$  strongly  in  
	 $L^{p(\cdot)}(\partial\O),$ 
	we have that $u=u_D $ on $\partial\Omega$, which proves 
	that $u\in \mathcal{A}$.

	Now, since $u_{h_j}\to u$ strongly in $L^2(\O)$ (due to 
	$s(x)\equiv2\in\mathcal{K}$)
	and 
	$\nabla u_{h_j}+R_h(u_{h_j})\rightharpoonup \nabla u$ 
	weakly in  $\lp$ we have
	$$
		J(u)\le\liminf_{j\to+\infty} \int_{\O}|\nabla 
		u_{h_j}+R(u_{h_j})|^{p(x)} 
		+ |u_{h_j}-\xi|^2 \, dx\leq \liminf_{j\to+\infty} 
		J_{h_j}(u_{h_j})
		\leq\limsup_{j\to+\infty} J_{h_j}(u_{h_j}).
	$$
	Let $v\in \mathcal{A}\cap W^{2,2}(\O)$ and $v_{h_j}$ 
	as in Lemma \ref{auxiliar}, we obtain
	$$
		J(u)\leq \liminf_{j\to+\infty} J_{h_j}(u_{h_j})
		\leq\limsup_{j\to+\infty}J_{h_j}(u_{h_j})
		\leq \lim_{j+\infty} J_{h_j}(v_{h_j})
		=J(v).
	$$
	By a density argument, we also have that 
	$$
		J(u)\leq \liminf_{j\to+\infty} J_{h_j}(u_{h_j})
		\leq\limsup_{j\to+\infty}J_{h_j}(u_{h_j})\leq J(w)
	$$
	for any $w\in  \mathcal{A}.$ Therefore $u$ is a 
	minimizer of $J.$ Moreover, if we take $w=u,$ 
	we have that $J_{h_j}(u_{h_j})\to J(u)$ as 
	$j\to+\infty.$ Thus, 
	\begin{equation*}
		\int_{\Gamma_{int}}|\[u_{h_j}\]|^{2}
		{\bf{h_j}}^{\nicefrac{-2}{p'(x)}}\, dS. 
		\to 0\quad\mbox{as }j\to+\infty,
	\end{equation*}
	Then, by Lemma \ref{lifcont}, we have that 
	$R_{h_j}(u_{h_j})\to 0$ as $j\to+\infty$ and
	$$
		\nabla u_{h_j}\rightharpoonup \nabla u \quad 
		\mbox{weakly in } L^{p(\cdot)}(\O).
	$$ 
	
	On the other hand, since
	$$
		\nabla u_{h_j}+R_{h_j}(u_h) \rightharpoonup\nabla u 
		\mbox{ weakly in } \lp \mbox{ and }
		\int_{\O}|\nabla u_{h_j}+R_h(u_{h_j})|^{p(x)}\, dx 
		\to \int_{\O}|\nabla u|^{p(x)}\, dx,
	$$
	by Proposition  \ref{debil}, 
	$\nabla u_{h_j}+R_{h_j}(u_{h_j})\to \nabla u$ strongly 
	in $\lp$. 
	Therefore, since $R_{h_j}(u_{h_j})\to 0$ strongly in
	$\lp$, we get that $\nabla u_{h_j}\to \nabla u$ 
	strongly in $\lp$.

	\medskip
	
	Finally, since $u$ is the unique minimizer, the hole 
	sequence converges.
\end{proof}

\begin{remark}
In the above theorem, we ask that $p_1>\nicefrac{N}2$ since in those cases we obtain $p^*(x)> 2.$ 
Observe that, when $N\in\{1,2\},$ we are not adding any new assumption on $p_1.$
\end{remark}
%%%%%%%%%%%%%%%%%%%%%%%%%%%%%%%%%%%%%%%%%%%%%%%%%%%%%%%%%%%%
%%%%%%%%%%%%%%%%%%%%%%%%%%%%%%%%%%%%%%%%%%%%%%%%%%%%%%%%%%%
\section{The decomposition--coordination method}\label{dcor}
In this section we will study the  
decomposition--coordination method to approximate,  
for each $h$, the minimizer of $J_h$.

\medskip

Throughout this section, to simplify notation, we omit the 
subindex $h,$ and $\l2\cdot\r2,$  $\|\cdot \|,$ and  
$\langle\cdot,\cdot\rangle$ denote the $L^2$-norm, 
$L^2\times L^2$-norm and  the inner product associated 
to the $L^2\times L^2$-norm, respectively.

\medskip

Let $V=S^k(\th)$ and  $H=S^l(\th)\times S^l(\th)$ where 
$k,l\in\mathbb{N}_0$ with $l\geq k-1$ we consider the 
following functionals
\begin{itemize}\renewcommand{\labelitemi}{\tiny$\blacktriangleright$}
	\item $F\colon H\to\R,$ 
	${\displaystyle F(q)\coloneqq\int_{\O}
	|q|^{p(x)}\, 
	dx};$\vspace{.3 cm}
	\item $G\colon V\to\R,$  
	\begin{math}\displaystyle G(v)
	\coloneqq\int_{\O}|v-\xi|^{2} 
	\, dx+\int_{\Gamma_{int}} |\[v\]|^2 
	{\bf{h}}^{-\nicefrac{2}{p^\prime(x)}}\, dS+ 
	\int_{\partial\O} |v-u_D|^2 {\bf{h}}^{-\nicefrac{2}
	{p^\prime(x)}}\, dS;\end{math}\vspace{.3 cm}
	\item $B\colon V\to H,$ 
	$\displaystyle Bv\coloneqq R(v)+\nabla v.$
\end{itemize}

Observe that 
$$
	J(v)=F(Bv)+G(v),
$$ 
$F$ and $G$ are convex and Gateaux--diffentiable 
functionals, $B$ is a linear operator, and
$$
	\dom(F\circ B)\cap \dom(G)\neq\emptyset.
$$

In \cite[Chapter VI]{G}, in a more general context, the 
author show that
the problem
\begin{equation}
	\min_{v\in V} J(v)=\min_{v\in V}\{F(Bv)+G(v)\}
	\label{P}
\end{equation}
is equivalent to
\begin{equation}
	\min_{(v,q)\in W} \{F(q)+G(v)\},
	\label{pi}
\end{equation}
where
$$
	W=\{(v,q)\in V\times H: Bv=q\}.
$$

\medskip
 
We then define for $r\ge0$ an augmented Lagrangian 
$\mathcal{L}_r$ associated with \eqref{pi}, by
$$
	\mathcal{L}_r\colon V\times H\times H\to \R
$$
$$
	\mathcal{L}_r(v,q,\lambda)\coloneqq 
	F(q)+G(v)+\langle\lambda,Bv-q\rangle+\frac{r}{2}\l2 Bv-
	q\r2^2,
$$
and we will say that  $(u,\eta,\lambda) \in V\times H\times 
H$ is a saddle point of $\mathcal{L}_r$ on $V\times H\times 
H$ if 
\begin{equation}\label{sadle}
 	\mathcal{L}_r(u,\eta,\mu)\le \mathcal{L}_r(u,\eta,	
 	\lambda)\le  \mathcal{L}_r(v,q,\lambda) 
	\quad \forall (v,q,\mu) \in V\times H\times H.
\end{equation}

The following lemma  establishes a fundamental relationship 
between the saddle points of $\mathcal{L}_r$ and the 
solution of \eqref{P}. For
the proof see \cite[Theorem 2.1-- Chapter VI]{G}.
\begin{lema}\label{sillasol} 
	Let $(u,\eta,\lambda)$ be a saddle point of $\mathcal{L}_r,$ on $V\times H\times H$ 
	then $u$ is the solution of \eqref{P} and $Bu=\eta.$
\end{lema}

Then, a method for solving \eqref{P} is to solve the saddle point problem \eqref{sadle}.

\medskip

\begin{remark}\label{caracterizacion}
Let $(u,\eta,\lambda)$ be a saddle point of $\mathcal{L}_r,$ then
$$
	\mathcal{L}_r(u,\eta,\lambda)\le  \mathcal{L}_r(v,q,\lambda) \quad 
	\forall (v,q,\mu) \in V\times H\times H, \quad (u,\eta)\in V\times H.
$$
Therefore $(u,\eta)$ is characterized by
\begin{align*}
	G(v)-G(u)+\langle\lambda, B(v-u)\rangle+ r\langle Bu-\eta,B(v-u)\rangle &\geq 0, \quad\forall {v\in V},\, u\in V,\\
	F(q)-F(\eta)-\langle\lambda, q-\eta\rangle + r\langle \eta-Bu,q-\eta\rangle &\geq 0, \quad\forall q\in H,\, \eta\in H.
\end{align*}
Moreover, since $F$ and $G$ are convex and Gateaux--diffentiable, $(u,\eta)$ is also characterized by
\begin{equation}\label{sadle2}
	\begin{aligned}
		G'(u)(v-u)+\langle \lambda, B(v-u)\rangle + r\langle Bu-\eta,B(v-u)\rangle &\geq 0, \quad\forall {v\in V},\, u\in V,\\
		F'(\eta)(q-\eta)-\langle \lambda, q-\eta\rangle+ r\langle\eta-Bu,q-\eta\rangle&\geq 0, \quad \forall q\in H,\, \eta\in H,
	\end{aligned}
\end{equation}
where $G'$ and $F'$ are the Gateaux-derivative of $G$ and $F,$ respectively.

For both characterizations, see \cite[Chapter I and VI]{G}.
\end{remark}

\subsection{Algorithms.}
To solve the saddle point problem \eqref{sadle} we will use an Uzawa type algorithm and a variant of it. See \cite{Cea1,GLT1,GLT3}.
\subsection*{Algorithm 1.}
Let $\lambda^0\in H$:  then $\lambda^n$ is known , we
define $(u^n, \eta^n,\lambda^{n+1})\in V\times H\times H$ by
\begin{align*}
	\mathcal{L}_r(u^n,\eta^n,\lambda^{n})\leq  
	\mathcal{L}_r(v,q,\lambda^{n}) \quad \forall (v,q)\in 
	V\times H,\\
	\lambda^{n+1}=\lambda^n +\rho_n(Bu^n- \eta^n),\quad 
	\rho_n>0.
\end{align*}

\begin{remark}\label{carcal}
	Observe that the first inequality of this algorithm is equivalent to the following system of two coupled variational 
	inequalities,
	\begin{equation}\label{sadle31}
		\begin{aligned}
			G(v)-G(u^n)+\langle \lambda^n, B(v-u^n)\rangle+ 
			r\langle Bu^n-\eta^{n},B(v-u^n)\rangle 
			&\geq 0, \quad \forall v\in V, u^n\in V,\\
			F(q)-F(\eta^n)-\langle\lambda^n, q-\eta^n\rangle+ 
			r\langle\eta^n-Bu^n,q-\eta^n\rangle&\geq 0, \quad \forall \, q\in H, \eta^n \in H,
		\end{aligned}
	\end{equation}
	see \cite[Chapter VI--Section 3]{G}.
\end{remark}

The main difficulty of Algorithm 1 is that it requires  the solution of the coupled system of equations at each iteration.
To overcome this difficulty, in \cite{G} the authors propose the following algorithm. 
\subsection*{Algorithm 2.}
Let $(\eta^0,\lambda^1)\in H\times H$; then $(\eta^{n-1},\lambda^n)$ known, we 
define $(u^n, \eta^n,\lambda^{n+1})\in V\times H\times H$ by
\begin{equation}\label{sadle3}
	\begin{aligned}
		G(v)-G(u^n)+\langle\lambda^n, B(v-u^n)\rangle+ r\langle Bu^n-\eta^{n-1},B(v-u^n)\rangle &\geq 0, \quad \forall v\in V, u^n\in V,\\
		F(q)-F(\eta^n)-\langle\lambda^n, q-\eta^n\rangle+ r\langle\eta^n-Bu^n,q-\eta^n\rangle&\geq 0, \quad \forall q\in H, \eta^n\in H,\\
\lambda^{n+1}=\lambda^n +\rho_n(Bu^n- \eta^n),\quad \rho_n>0.
\end{aligned}
\end{equation}
Observe that now the two first equations are uncoupled. 

\medskip

\begin{remark}\label{lapapa2}
In \cite{G}, in a more general context, the convergence of both algorithms are proved.
More precisely, if $F,$ $G$ and $B$ satiefy the assumptions (H1), (H2) and (H3), then
\begin{equation}\label{glow1}
	\begin{aligned}
		\lim_{n\to+\infty} \| B u^{-n}-\eta^{-n}\|&= 0,\\
		\lim_{n\to+\infty} \langle F'(\eta^n)-F'(\eta),\eta^n-\eta\rangle&=0,\\
		\lim_{n\to+\infty} \|B u^{n}-\eta^{n}\|&= 0,\\
		\lim_{n\to+\infty} \|\eta^{n}-\eta\|&= 0,\\
		 Bu^{n}\to \eta=Bu & \mbox{ in } H,\\
		 \lambda^{n+1}-\lambda^n \to 0 &\mbox{  in } H,\\
		\lambda^n \mbox{ is } \mbox{bounded, }&\quad
	\end{aligned}
\end{equation}
see Theroem 4.1 and Theorem 5.1 in \cite[Chapter VI]{G}.
The assumption (H4) is only used to concluded that $u^n\to u$  in $V.$ 

In our case, (H4) does not hold, that is $B$ is not 
injective.  
To overcome the lack of this assumption, we use that our 
functional 
$G$ is Gateaux-differentiable and convex.
\end{remark}

\begin{proof}[Proof of Theorem \ref{conalg1}]
By Remark \ref{carcal} and using the same argument of Remark \ref{caracterizacion}, $u^n$ 
can be characterized as
\begin{equation}\label{sadle33}
	%\begin{aligned}
		G'(u^n)(v-u^n)+\langle\lambda^n, B(v-u^n)\rangle+ r\langle Bu^n-\eta^n,B(v-u^n)\rangle\geq 0, \quad \forall v\in V,\ u^n \in V,
		%F'(\eta^n)(q-\eta^n)-\langle\lambda^n, q-\eta^n\rangle + r\langle\eta^n-Bu^n,q-\eta^n\rangle&\geq 0, \quad \forall q\in H,\ \eta^n\in H.
	%\end{aligned}
\end{equation}

Let us denote $u^{-n}=u^n-u$  and $\eta^{-n}=\eta^n-\eta$. By Remark \ref{lapapa2}, we have that \eqref{glow1} holds. 
%\begin{equation}\label{glow1}
%	\begin{aligned}
%		\lim_{n\to+\infty} \| B u^{-n}-\eta^{-n}\|&= 0,\\
%		\lim_{n\to+\infty} \langle F'(\eta^n)-F'(\eta),\eta^n-\eta\rangle&=0,\\
%		\lim_{n\to+\infty} \|B u^{n}-\eta^{n}\|&= 0,\\
%		\lim_{n\to+\infty} \|\eta^{n}-\eta\|&= 0,\\
%		 Bu^{n}\to \eta=Bu & \mbox{ in } H,\\
%		 \lambda^{n+1}-\lambda^n \to 0 &\mbox{  in } H,\\
%		\lambda^n \mbox{ is } \mbox{bounded, }&\quad
%	\end{aligned}
%\end{equation}

On the other hand, taking $v=u^n$ in \eqref{sadle2}, $v=u$ in \eqref{sadle33} and summing we obtain
$$
	(G'(u^n)-G'(u))(u-u^n)+\langle\lambda^n-\lambda, 
	B(u-u^n)\rangle
	+r\langle B(u^n-u)+\eta-\eta^n,B(u-u^n)\rangle\ge 0,
$$
then
%\begin{equation}\label{glow}
$$
	(G'(u^n)-G'(u))(u^n-u)+\langle\lambda^n-\lambda, 
	B(u^n-u)\rangle+ r\langle B(u^n-u)-(\eta^n-\eta),
	B(u^n-u)\rangle
	\leq 0.
$$
%\end{equation}

Since $(G'(u^n)-G'(u))(u^n-u)\geq 0$, by \eqref{glow1}, we get
\begin{align*}
	(G'(u^n)-G'(u))(u^n-u)\!\!
	&=\!\!\int_{\O}\!\! |u^n-u|^{2}dx\!+\! \int_{\Gamma_{int}}\!\!\! |\[u^n-u\]|^2 {\bf{h}}^{-\nicefrac{2}{p^\prime(x)}}dS
	\!+\!
	\int_{\partial\O}\!\!\! |u^n-u|^2 {\bf{h}}^{-\nicefrac{2}{p^\prime(x)}}dS\\
	&\to 0\quad\mbox{as }n\to+\infty,
\end{align*}
then
\begin{equation}\label{conldos}
	\begin{aligned}
		u^n \to u &\mbox{  in } V,\\
		\[u^n-u\]\to 0 &\mbox{  in } \ldosg,\\
		u^n \to u &\mbox{  in } \ldosb.
	\end{aligned}
\end{equation}

Finally,
\begin{equation}\label{conldosd}
	\nabla u^{-n} \to 0 \quad \mbox{ in } H.
\end{equation}
due to
\begin{align*}
	\|R(u^{-n})\|_{\ldos}&\leq C \|{\bf{h}}^{-1/2}\[u^{-n}\]\|_{\ldosg}\to 0,\\
	B(u^{-n})&\to 0 \quad \mbox{  in } H.
\end{align*}
The proof is now completed.
\end{proof}

Finally, we prove the convergence of Algorithm 2.

\begin{proof}[Proof of Theorem \ref{conalg1}]
We began by observing that, as in the proof of Theorem \ref{conalg1}, by Remark \ref{lapapa2}, we obtain that \eqref{glow1} holds.

On the other hand, for this algorithm, we get that $u^n$ satisfies that
\begin{equation}\label{sadle42}
	G'(u^n)(v-u^n)+\langle \lambda^n, B(v-u^n)
	\rangle+ r\langle Bu^n-\eta^{n-1},B(v-u^n)
	\rangle\geq 0, \quad \forall v\in V
\end{equation}
Therefore, taking $v=u^n$ in \eqref{sadle2}, $v=u$ in \eqref{sadle42} and fallowing the lines of the proof of Theorem \ref{conalg1}, we get
\eqref{conldos} and \eqref{conldosd}. The proof is now completed.
\end{proof}
%%%%%%%%%%%%%%%%%%%%%%%%%%%%%%%%%%%%%%%%%%%%%%%%%%%%%%%%%%%%
%%%%%%%%%%%%%%%%%%%%%%%%%%%%%%%%%%%%%%%%%%%%%%%%%%%%%%%%%%%%
\section{Numerical Results} \label{numerical}
In this section, we only implement the uncoupled Algorithm 2. For any $h,$ 
we obtain a sequence $\{u^n_h\}$ such that $u^n_h\to u^h$ as $n\to+\infty,$ 
where  $u_h$ is the minimizer of $J_h$.

\medskip

For simplicity, we take
$$
	J_h(v)=F(Bv)+G(v),
$$
where
\begin{align*}
	F(q)&=\int_{\O} \dfrac{|q|^{p(x)}}{p(x)} \, dx,\\
	G(v)&=\frac{1}{2}\left(\int_{\O}|v-\xi|^{2}\, dx
	+\int_{\partial\O}|v-u_D|^{2}{\bf{h}}^{-\nicefrac{2}{p^\prime(x)}}\,dS
	+\int_{\Gamma_{int}}|\[v\]|^{2}{\bf{h}}^{-\nicefrac{2}{p^\prime(x)}}\, 
	dS\right).
\end{align*}

Observe that, this new definition of  $F$ does not change any of the results 
that we prove in the preceding sections.

\medskip

If we take $\rho_n=r$ then the algorithm is: 

Given 
$$
(\eta^0,\lambda^1)\in H \times H ,
$$
then, $(\eta^{n-1}_h,\lambda^n_h)$ known, we define $(u^{n}_h,\eta^{n}_h,\lambda^{n+1}_h)\in V\times H \times H$ by
\begin{equation}
	\begin{aligned}
	\int_{\O} (u^n_h-\xi) v\, dx +r\int_{\O} \left(Bu^n_h-\eta^{n-1}_h\right) Bv\, dx
	&+\int_{\O} \lambda^n_h Bv \, dx\\
	+\int_{\Gamma_{int}} \[u^n_h\] \[v\] {\bf{h}}^{-\nicefrac{2}{p'(x)}} dS 
	&+\int_{\partial\O}  (u^n_h-u_D) v {\bf{h}}^{-\nicefrac{2}{p'(x)}}dS= 0,
	\end{aligned}
	\label{M1}
\end{equation}
\begin{align}
	\label{M2} \int_{\O}|\eta^n_h|^{p(x)-2} \eta^n_h \Phi\,dx& 
	+r\int_{\O}(\eta^n_h-Bu^n_h)\Phi \, dx
	=\int_{\O} \lambda^n_h \Phi\, dx,\\
	\label{M3}\lambda^{n+1}_h&=\lambda^n_h +r(Bu^n_h- \eta^n_h),
 \end{align}
for all $v\in S^k(\th)$ and for all $\Phi\in S^l(\th)\times S^l(\th).$

\begin{remark}
Since $V,H,F,G,B,\rho_n$ and $r$  satisfy the assumptions of  Theorem \ref{conalg2},  
then the conclusions of Theorem \ref{conalg2} hold, that is,
$u^n_h \to u_h$ and $\nabla u^n_h \to \nabla u_h,$ as $n\to+\infty$.
\end{remark}

Observe that \eqref{M1} can be replace by,
$$MU^n=F^n,$$
where 
\begin{align*}
	M_{ij}&=\int_{\O} \varphi_i \varphi_j\, dx +r\int_{\O} B\varphi_i B\varphi_j\, dx+
	\int_{\Gamma_{int}} \[\varphi_i\] \[\varphi_i\] {\bf{h}}^{-\nicefrac{2}{p^\prime(x)}}\, dS+
	\int_{\partial\O}  \varphi_i \varphi_j {\bf{h}}^{-\nicefrac{2}{p^\prime(x)}}\, dS,\\
	F^n_j&=\int_{\O}\varphi_j \xi \, dx+\int_{\partial\O} \varphi_j u_D {\bf{h}}^{-\nicefrac{2}{p^\prime(x)}}\, dS+ 
	\int_{\O} (r \eta^{n-1}_h-\lambda^n_h) B\varphi_j\, dx,
\end{align*}
and $\{\varphi_j\}_{j\leq m}$ is a basis of $V$ with $m=\mbox{dim}(V)$. Thus
$$
	u^n_h=\sum_{j=1}^m U^{n}_j	 \varphi_j.
$$

On the other hand, we  define
$\eta_{n,{\kappa}}=\eta^n_h|_{\kappa}$, in the same way we define $\lambda_{n,{\kappa}}$ and $B_{\kappa}u^n_h$. 
We can see from  \eqref{M2} that $\eta_{n,{\kappa}}$  satisfies
$$
	\left(\dfrac{1}{|\kappa|}\int_{\kappa}|\eta_{n,{\kappa}}|^{p(x)-2}\,dx +r\right)\eta_{n,{\kappa}}
	=\lambda_{n,{\kappa}}+ B_\kappa u^n_h.
$$

Let $\bar{p}_{\kappa}=p(\bar{x}_{\kappa})$, where $\bar{x}_{\kappa}$  is the varicenter of $\kappa$. Then using a 
quadrature rule for the first term, we can approximate $\eta_{n,{\kappa}}$ by the equation,
$$
(|\eta_{n,{\kappa}}|^{\bar{p}_{\kappa}-2}+r)\eta_{n,{\kappa}}=\lambda_{n,{\kappa}}+ B_\kappa u^n_h, 
$$
thus $|\eta_{n,{\kappa}}|$ solves
$$
(|\eta_{n,{\kappa}}|^{\bar{p}_{\kappa}-2}+r)|\eta_{n,{\kappa}}|=|\lambda_{n,{\kappa}}+B_{\kappa}u^n_h|,
$$
and therefore
$$
\eta_{n,{\kappa}}=\frac{\lambda_{n,{\kappa}}+B_{\kappa}u^n_h}{|\eta_{n,{\kappa}}|^{\bar{p}_{\kappa}-2}+r}.
$$

Summarizing, each iteration of the algorithm  can be reduced to the 
following:

\medskip

Find $(u^{n}_h,\eta^{n}_h,\lambda^{n+1}_h)\in V\times H \times H$ such that
\begin{align*}
	u^n_h&=\sum_{j=1}^m U^n_{j} \varphi_j,\\
	\eta_{n,{\kappa}}&=\frac{\lambda_{n,{\kappa}}
	+B_{\kappa}u^n_h}{x^{\bar{p}_{\kappa}-2}+r},\\
	\lambda^{n+1}_h&=\lambda^n_h +r(B u^n_h- \eta^n_h).
\end{align*}
where $U^n$ solves,
\begin{equation}\label{lineal}
	MU^n=F^n
\end{equation}
and $x\in\R_{\geq 0}$ solves  
\begin{equation}\label{nolineal}
	x^{\bar{p}_{\kappa}-1}+rx=|\lambda_{n,{\kappa}}+B_{\kappa}u^n_h|,
\end{equation}

Observe that each step of the algorithm consists in solving the linear equation \eqref{lineal} and then the one dimensional nonlinear equation \eqref{nolineal}.

\medskip

We now apply the algorithm to a family of examples. For each $h$, we approximate $u_h$ by $u^n_h$, 
and finally  we compute $\|u^n_h-u\|_{L^{2}(\Omega)}$.

\medskip

Motivate by \cite{SW}, where the authors analyse a $P_0$ discontinuous Galerkin formulation for image denosing, we test this algorithm in the following example; we have considered a rectangular domain 
$\Omega=[-1\ 1]\times [-1\ 1]$ and a uniform rectangular mesh, with  
constant finite elements in all the rectangles.
We  denote by $m$ the number 
of degrees of freedom in the finite element approximation.
We take $r=1.$ 

%\begin{figure}[h]
% \includegraphics[width=1\textwidth]{mesh.eps}
 % \caption{Mesh with $N=100$}
 % \label{fig:figura0}
%\end{figure}

We take the following   function $p(x)$, 

$$
	p(x)=
	\begin{cases} 
		1+\left({\dfrac{b}{2}(x_1+x_2)+1+b}\right)^{-1} 
		&\mbox{ if } b\neq 0,\\
		2 &\mbox{ if } b =0.
	\end{cases}
$$

Observe that $p_2=2$ and $p_1=1+\nicefrac{1}{1+2b}$ , then $p_1$ is close to one when $b>>0$.

%\begin{figure}[h]
 %\includegraphics[width=1\textwidth]{agunasp2.eps}
%  \caption{Function $p(x)$ for difrent values of $b$}
%  \label{fig:figura4}
%\end{figure}

It is easy to see that the solution of \eqref{problema} is
$$
	u(x)=
	\begin{cases} 
		\dfrac{\sqrt{2}e^{b+1}}b
		\left(e^{\frac{b}{2}(x_1+x_2)}-1\right) 
		&\mbox{ if } b \neq 0,\\[.5cm] 
		\dfrac{\sqrt{2}e}{2}  (x_1+x_2) &\mbox{ if } b =0.
	\end{cases}
$$
with $\xi=u.$ 

%\begin{figure}[h]
 %\includegraphics[width=1\textwidth]{funcionp.eps}
 % \caption{Function $p(x)$ for $b=3$}
 % \label{fig:figura2}
%\end{figure}
%\begin{figure}[h]
% \includegraphics[width=1\textwidth]{variasu.eps}
 % \caption{Solution for different values of $b$}
 % \label{fig:figura3}
%\end{figure}
 
The experimental results for different values of $b$ and $m$ are shown in the following table. 

\begin{center}
\begin{tabular}{ | c | c | c | c | c | c | c | }
	\hline   
   b & \multicolumn{2}{|c|}{0} & \multicolumn{2}{|c|}{0.25} & \multicolumn{2}{|c|}{0.5} \\
   \hline
   m & $L^2-$Error & Iter. & $L^2-$Error & Iter. & $L^2-$Error & Iter. \\
   \hline
   \hline
   100 & 0.5921  & 3 & 0.7519 & 46 & 0.9214 &  50\\
   \hline
   196 & 0.4603  & 3 & 0.5932 & 47 & 0.7313 &  52\\
   \hline   
   484 & 0.3185  & 3 & 0.4220 & 48 & 0.5271 &  56\\
   \hline   
   961 & 0.2366  & 3 & 0.3228 & 49 & 0.4087 &  59\\
   \hline
   2916 & 0.1430 & 3 & 0.2101 & 50 & 0.2744 &  63\\
\hline
 \end{tabular}
 \end{center}
 
 \medskip

 Where Iter. is the number of iterations required
 in the algorithm in order to satisfy our stopping 
 criteria. Observe that as $b$ grows, the number of iterations increases and the rate of convergence decreases.

\bigskip 
\begin{center}
 \begin{tikzpicture}[font= \tiny,scale=1.2]
    \begin{loglogaxis}[
        xlabel=\textsc{Total DOF},           
        ylabel={$L^2-$ Error}
        ]
      
      \addplot [mark=*, color=black] plot coordinates {
        (100,  5.921e-01 )
        (196,  4.603e-01 )
        (484,  3.185e-01 )
        (961,  2.366e-01 )
        (2916, 1.430e-01 )

    };

    \addplot[mark=triangle*, color=black] plot coordinates {
        (100, 0.7519 )
        (196, 0.5932 )
        (484, 0.4220 )
        (961, 0.3228 )
        (2916,0.2101 )
    };

    \addplot [mark=square*, color=black] plot coordinates {
        (100, 0.9214 )
        (196, 0.7313 )
        (484, 0.5271 )
        (961, 0.4087 )
        (2916,0.2744 )
    };
    \legend{$b=0$\\$b=0.25$\\$b=0.5$\\}
    \end{loglogaxis};
\end{tikzpicture}
\end{center}

\bibliographystyle{amsplain}
\bibliography{las}
\end{document}